\documentclass{article}%
\usepackage{amsmath}
\usepackage{amsfonts}
\usepackage{amssymb}
\usepackage{graphicx}%
\providecommand{\U}[1]{\protect\rule{.1in}{.1in}}
\newtheorem{theorem}{Theorem}

\newtheorem{proposition}[theorem]{Proposition}

\newenvironment{proof}[1][Proof]{\noindent\textbf{#1.} }{\ \rule{0.5em}{0.5em}}
\begin{document}

\title{ }

\begin{center}
{\LARGE New series with Cauchy and Stirling numbers, Part 2}\bigskip

Khristo N. Boyadzhiev

Department of Mathematics, Ohio Northern University Ada, Ohio 45810, USA

\textbf{E-mail: k-boyadzhiev@onu.edu}\bigskip

Levent Karg\i n

Department of Mathematics, Akdeniz University, TR-07058 Antalya, Turkey

\textbf{E-mail: lkargin@akdeniz.edu.tr}\bigskip\bigskip

\textbf{Abstract}
\end{center}

\begin{quotation}
We evaluate in closed form several series involving products of Cauchy numbers
with other special numbers (harmonic, skew-harmonic, hyperharmonic, and
central binomial). Similar results are obtained with series involving Stirling
numbers of the first kind. We focus on several particular cases which give new
closed forms for Euler sums of hyperharmonic numbers and products of
hyperharmonic and harmonic numbers.
\end{quotation}



\section{Introduction}

The Cauchy numbers $c_{n}$ are defined by the generating function%

\[
\frac{x}{\ln(x+1)}=\sum_{n=0}^{\infty}\frac{c_{n}}{n!}x^{n}\quad(|x|<1)
\]
(see \cite{B,COMTET,L}). They are called Cauchy numbers of the first kind by
Comtet \cite{COMTET}. The numbers $c_{n}/n!$ are also known as the Bernoulli
numbers of the second kind (see the comments in \cite{B}). The Cauchy numbers
have the important representation
\begin{equation}
c_{n}=\int\limits_{0}^{1}z(z-1)\cdots(z-n+1)\,dz \label{B1}%
\end{equation}
The Stirling numbers of the first kind $s(n,k)$ are defined by the ordinary
generating function
\[
z(z-1)\cdots(z-n+1)=\sum_{k=0}^{n}s(n,k)\,z^{k}%
\]
or, equivalently
\begin{equation}
n!\binom{z}{n}=\sum_{k=0}^{n}s(n,k)\,z^{k} \label{B2}%
\end{equation}
and together with the Cauchy numbers play a major role in this paper.
Integrating equation (\ref{B1}) we see that the numbers $c_{n}$ can be
expressed in terms of $s(n,k)$ in the following way
\begin{equation}
c_{n}=\sum_{k=0}^{n}\frac{s(n,k)}{k+1}. \label{L1}%
\end{equation}
The Stirling numbers of the first kind are very popular numbers in mathematics
and have various important applications (see the comments and references in
\cite{B,COMTET}). In the recent paper \cite{B} the first author stated the
following two propositions:\medskip

\textbf{Proposition A: }\textit{Let }$f(z)$\textit{ be a function analytic in
a region of the form }$\operatorname{Re}(z)>\lambda$\textit{ for some
}$\lambda<0$\textit{ and with moderate growth in that region. Then we have the
representation }%
\begin{equation}
\int\limits_{0}^{1}f(x)dx=\sum_{n=0}^{\infty}\frac{(-1)^{n}c_{n}}{n!}\left\{
\sum_{k=0}^{n}\binom{n}{k}(-1)^{k}f(k)\right\}  . \label{B3}%
\end{equation}

\textbf{Proposition B:}\textit{ Under the same assumptions on the
function}$\ f(z)$\textit{, for every }$m\geq0$\textit{ we have the
representation }%
\begin{equation}
\frac{f^{(m)}(0)}{m!}=\sum_{n=0}^{\infty}\frac{(-1)^{n}s(n,m)}{n!}\left\{
\sum_{k=0}^{n}\binom{n}{k}(-1)^{k}f(k)\right\}  \label{B4}%
\end{equation}
\textit{and in particular, }%
\begin{equation}
f^{^{\prime}}(0)=\sum_{n=1}^{\infty}\frac{1}{n}\left\{  \sum_{k=0}^{n}%
\binom{n}{k}(-1)^{k-1}f(k)\right\}  . \label{B5}%
\end{equation}
(The summation in (\ref{B4}) de facto starts from $n=m$ since $s(n,m)=0$ for
$n<m$.)

For details see \cite{B}. In that paper various series identities were proved
based on these two propositions by applying them to appropriate functions. The
purpose of the present paper is to continue this project and present further
results in this direction.

In the next section we prove new series identities involving Cauchy numbers
and binomial coefficients. The short Section 3 deals with skew-harmonic
numbers, while Section 4 is dedicated to series with hyperharmonic numbers.
Our results are presented in several examples and propositions.

\section{Series with Cauchy numbers and binomial coefficients}

\textbf{Example 1: }In this example we construct the generating functions for
the numbers $c_{n}\binom{n}{q}$ and $s\left(  n,m\right)  \binom{n}{q}$ for
any integer $q\geq0$.

\begin{theorem}
{For any non-negative integer }$q${ and every }$\left\vert z\right\vert
<1${${\ }$we have the representation }
\begin{align}
\sum_{n=0}^{\infty}\frac{(-1)^{n}c_{n}}{n!}\binom{n}{q}z^{n}  &
=(-1)^{q}\left(  \frac{z}{1-z}\right)  ^{q}\int\limits_{0}^{1}\binom{x}%
{q}(1-z)^{x}dx\label{B6}\\
&  =(-z)^{q}\int\limits_{0}^{1}\binom{x}{q}(1-z)^{x-q}dx,\nonumber
\end{align}
where
\[
\int\limits_{0}^{1}\binom{x}{q}(1-z)^{x}dx=\frac{1}{q!}\sum_{k=0}%
^{q}s(q,k)A_{k}%
\]
with
\[
A_{k}=k!\left(  \frac{1}{(-\ln(1-z))^{k+1}}-(1-z)\sum_{j=0}^{k}\frac
{1}{j!(-\ln(1-z))^{k-j+1}}\right)  .
\]
In particular, with $z=1/2$ we have
\begin{equation}
\sum_{n=0}^{\infty}\frac{(-1)^{n}c_{n}}{n!2^{n}}\binom{n}{q}=\frac{(-1)^{q}%
}{q!}\sum_{k=0}^{q}s(q,k)A_{k} \label{B7}%
\end{equation}
with
\[
A_{k}=\int\limits_{0}^{1}x^{k}2^{-x}dx=k!\left(  \frac{1}{(\ln2)^{k+1}}%
-\frac{1}{2}\sum_{j=0}^{k}\frac{1}{j!(\ln2)^{k-j+1}}\right)  .
\]

\end{theorem}

\begin{proof}
For the proof we use the binomial formula (see \cite[Eq.(10.25)]{B2018})
\begin{align}
\sum_{k=0}^{n}\binom{n}{k}(-1)^{k}\binom{k}{q}\alpha^{k}  &  =(-\alpha
)^{q}(1-\alpha)^{n-q}\binom{n}{q}\nonumber\\
&  =\left(  \frac{-\alpha}{1-\alpha}\right)  ^{q}(1-\alpha)^{n}\binom{n}{q}
\label{B7a}%
\end{align}
where $0\leq\alpha\leq1$. We take $f(x)=\binom{x}{q}\alpha^{x}$ to get from
(\ref{B3})
\begin{equation}
\int\limits_{0}^{1}\binom{x}{q}\alpha^{x}dx=\left(  \frac{-\alpha}{1-\alpha
}\right)  ^{q}\sum_{n=0}^{\infty}\frac{(-1)^{n}c_{n}(1-\alpha)^{n}}{n!}%
\binom{n}{q}. \label{B8}%
\end{equation}
When $\alpha=1$ this becomes the familiar
\[
c_{q}=q!\int\limits_{0}^{1}\binom{x}{q}dx.
\]
Now equation (\ref{B6}) follows from (\ref{B8}) with $z=1-\alpha$. For the
evaluation of the integral we use equation (\ref{B2}) and entry 3.352(1) from
\cite{GR}.
\end{proof}

Next, applying Proposition B to the same function $f(x)=\binom{x}{q}\alpha
^{x}$ and using again the binomial identity (\ref{B7a}) we come to the
following result.

\begin{proposition}
{For every non-negative integer }$q${ and every }$0\leq z<1${ we have for
}$m=1,2,...$
\begin{equation}
\sum_{n=0}^{\infty}\frac{(-1)^{n}s(n,m){z}^{n}}{n!}\binom{n}{q}=\frac
{(-1)^{q}}{m!}\left(  \frac{z}{1-z}\right)  ^{q}\left(  \frac{d}{dx}\right)
^{m}\left.  (1-z)^{x}\binom{x}{q}\right\vert _{x=0}. \label{B9}%
\end{equation}

\end{proposition}

For $m=1$ using the formula
\[
\frac{d}{dx}\binom{x}{q}=\binom{x}{q}\sum_{j=0}^{q-1}\frac{1}{x-j}=\binom
{x}{q}\frac{1}{x}+\binom{x}{q}\sum_{j=1}^{q-1}\frac{1}{x-j}%
\]
we compute
\[
\lim_{x\rightarrow0}\left\{  \frac{d}{dx}(1-z)^{x}\binom{x}{q}\right\}
=\frac{(-1)^{q-1}}{q}.
\]
Also $s(n,1)=(-1)^{n-1}(n-1)!$ so that (\ref{B9}) takes the form
\[
\sum_{n=0}^{\infty}\frac{z^{n}}{n}\binom{x}{q}=\frac{1}{q}\left(  \frac
{z}{1-z}\right)  ^{q}%
\]
which is equivalent to the well-known expansion
\[
\sum_{n=0}^{\infty}\binom{n}{q}z^{n}=\frac{z^{q}}{(1-z)^{q+1}}.
\]

\textbf{Example 2: }In this example we use the central binomial coefficients
$\binom{2n}{n}$. We start with the binomial formula \cite[Eq. (10.35a)]%
{B2018}
\begin{equation}
\sum_{k=0}^{n}\binom{n}{k}(-1)^{k}\binom{2k}{k}\frac{1}{4^{k}}=\binom{2n}%
{n}\frac{1}{4^{n}}. \label{B10}%
\end{equation}
So we apply (\ref{B3}) to the function
\[
f(x)=\binom{2x}{x}\frac{1}{4^{x}}=\frac{\Gamma(2x+1)}{\Gamma^{2}(x+1)4^{x}}%
\]
to get from (\ref{B3}) and (\ref{B4})

\begin{proposition}%
\begin{equation}
\sum_{n=0}^{\infty}\,\frac{(-1)^{n}c_{n}}{n!4^{n}}\binom{2n}{n}=\int%
\limits_{0}^{1}\binom{2x}{x}\frac{1}{4^{x}}dx\approx0\mathrm{.6703837612}
\label{B11}%
\end{equation}
{and also for }$m=1,2,...$
\begin{equation}
\frac{1}{m!}\left.  \left(  \frac{d}{dx}\right)  ^{m}\binom{2x}{x}\frac
{1}{4^{x}}\right\vert _{x=0}=\sum_{n=1}^{\infty}\frac{(-1)^{n}s(n,m)}{n!4^{n}%
}\binom{2n}{n}. \label{B12A}%
\end{equation}

\end{proposition}

For $m=\,1,2,3$ we have correspondingly
\begin{align}
\sum_{n=1}^{\infty}\frac{1}{\,n4^{n}}\binom{2n}{n}  &  =\ln4\label{B13}\\
\sum_{n=1}^{\infty}\frac{H_{n-1}}{n4^{n}}\binom{2n}{n}  &  =\frac{\pi^{2}}%
{6}+2\ln^{2}2\label{B14}\\
\sum_{n=1}^{\infty}\frac{(H_{n-1}^{2}-H_{n-1}^{(2)})}{n4^{n}}\binom{2n}{n}  &
=4\zeta(3)+\frac{8}{3}\ln^{3}(2)+\frac{2\pi^{2}}{3}\ln2 \label{B15}%
\end{align}
as $s(n,2)=(-1)^{n-2}(n-1)!H_{n-1}$ and $s(n,3)=(-1)^{n-3}\frac{(n-1)!}%
{2}(H_{n-1}^{2}-H_{n-1}^{(2)})$. Here%
\[
H_{k}=1+\frac{1}{2}+\cdots+\frac{1}{k}\quad(k\geq1),\quad H_{0}=0
\]
and
\[
H_{k}^{(2)}=1+\frac{1}{2^{2}}+\cdots+\frac{1}{k^{2}}\quad(k\geq1),\quad
H_{0}^{(2)}=0.
\]
The first series (\ref{B13}) is known (see \cite[Eq.(6)]{LEHMER}). All these
series are very slowly convergent.

\textbf{Example 3: }The starting point is the binomial identity
(\cite[Eq.(10.9)]{B2018})%
\[
\sum_{k=0}^{n}\left(  -1\right)  ^{k}\binom{n}{k}\binom{p+k}{k}=\left(
-1\right)  ^{n}\binom{p}{n},
\]
where $p\geq0$ is an integer. With the function $f\left(  x\right)
=\binom{p+x}{p}$ we find from (\ref{B3}) the identity
\begin{equation}
\int\limits_{0}^{1}\binom{p+x}{p}dx=\sum_{n=0}^{p}\frac{c_{n}}{n!}\binom{p}%
{n}. \label{B18}%
\end{equation}
We have
\begin{align*}
\binom{p+x}{p}  &  =\frac{\left(  p+x\right)  \left(  p+x-1\right)
\cdots\left(  x+1\right)  }{p!}\\
&  =\left(  -1\right)  ^{p+1}\frac{\left(  -x\right)  \left(  -x-1\right)
\cdots\left(  -x-\left(  p+1\right)  +1\right)  }{p!}\\
&  =\left(  -1\right)  ^{p+1}\sum_{k=0}^{p+1}\left(  -1\right)  ^{k-1}%
s(p+1,k)\,x^{k-1}.
\end{align*}
Then (after changing the index $k=j+1$) we get
\begin{align*}
\int\limits_{0}^{1}\binom{p+x}{p}dx  &  =\left(  -1\right)  ^{p+1}\sum
_{k=1}^{p+1}\frac{\left(  -1\right)  ^{k-1}s(p+1,k)}{k}\\
&  =\left(  -1\right)  ^{p+1}\sum_{j=0}^{p}\frac{\left(  -1\right)
^{j}s(p+1,j+1)}{j+1}.
\end{align*}
Using the explicit expression of the Cauchy polynomials of the second kind
$\hat{c}_{n}\left(  x\right)  $ \cite[Theorem 2]{KM}%
\[
\hat{c}_{n}\left(  -r\right)  =\sum_{k=0}^{n}s_{r}\left(  n,k\right)
\frac{\left(  -1\right)  ^{k}}{k+1},
\]
where $s_{r}\left(  n,k\right)  $ is the $r$-Stirling numbers of the first
kind \cite{BRODER}, and $s_{1}\left(  n,k\right)  =s\left(  n+1,k+1\right)  ,$
we obtain
\begin{equation}
\sum_{n=0}^{p}\frac{c_{n}}{n!}\binom{p}{n}=\frac{\left(  -1\right)  ^{p}}%
{p!}\hat{c}_{p}\left(  -1\right)  . \label{L14}%
\end{equation}

With the function $f\left(  x\right)  =\binom{p+x}{p}$, (\ref{B4}) implies%
\[
\frac{1}{m!}\left(  \frac{d}{dx}\right)  ^{m}\left.  \binom{p+x}{p}\right\vert
_{x=0}=\sum_{n=m}^{p}\frac{s\left(  n,m\right)  }{n!}\binom{p}{n}.
\]
Moreover, it is known that \cite[Theorem 2]{WJ} ($n\geq m>0$)%
\begin{align*}
&  \left(  \frac{d}{dx}\right)  ^{m}\left.  \binom{n+x}{r}\right\vert _{x=0}\\
&  \qquad\quad=\left(  -1\right)  ^{m}\binom{n}{r}Y_{m}\left(  -0!\left(
H_{n}-H_{n-r}\right)  ,\ldots,-\left(  m-1\right)  !\left(  H_{n}^{\left(
m\right)  }-H_{n-r}^{\left(  m\right)  }\right)  \right)  ,
\end{align*}
where $H_{n}^{\left(  m\right)  }$ is the $n$th generalized harmonic number
defined by
\[
H_{n}^{\left(  m\right)  }=1+\frac{1}{2^{m}}+\cdots+\frac{1}{n^{m}}\quad
(m\geq1),\quad H_{0}^{(m)}=0
\]
and $Y_{i}\left(  t_{1},t_{2},\ldots t_{i}\right)  $ is the exponential
complete Bell polynomial \cite[ Sect. 3.3]{COMTET}. Thus, we have%
\begin{equation}
\sum_{n=m}^{p}\frac{s\left(  n,m\right)  }{n!}\binom{p}{n}=\frac{\left(
-1\right)  ^{m}}{m!}Y_{m}\left(  -0!H_{p},\ldots,-\left(  m-1\right)
!H_{p}^{\left(  m\right)  }\right)  . \label{L15}%
\end{equation}
The following proposition summarizes these results

\begin{proposition}
For every non-negative integers $m$ and $p,$ (\ref{L14}) and (\ref{L15}) are true.
\end{proposition}

For $m=1,$ (\ref{L15}) reduces to \cite[Eq. (9.2)]{B2018}%
\[
\sum_{n=1}^{p}\binom{p}{n}\frac{\left(  -1\right)  ^{n+1}}{n}=H_{p}.
\]
For $m=2$ and $m=3,$ we find correspondingly%
\begin{align*}
\sum_{n=1}^{p}\left(  -1\right)  ^{n+1}\binom{p+1}{n+1}\frac{H_{n}}{n+1}  &
=\frac{H_{p+1}^{2}-H_{p+1}^{\left(  2\right)  }}{2}\\
\sum_{n=1}^{p}\left(  -1\right)  ^{n-1}\binom{p+2}{n+2}\frac{H_{n+1}%
^{2}-H_{n+1}^{(2)}}{n+2}  &  =\frac{H_{p+2}^{3}-3H_{p+2}H_{p+2}^{\left(
2\right)  }+H_{p+2}^{\left(  3\right)  }}{3}.
\end{align*}

\textbf{Example 4: }Here we give a new explicit formula, extending (\ref{L1})
and \cite[Eq.(17)]{L}. If $q>n,$ $\binom{n}{q}=0.$ Then (\ref{B6}) can be
rewritten as
\begin{align*}
q!\int\limits_{0}^{1}\binom{x}{q}(1-z)^{x}dx  &  =\sum_{n=q}^{\infty}%
\frac{(-1)^{n-q}c_{n}}{\left(  n-q\right)  !}z^{n-q}\left(  1-z\right)  ^{q}\\
&  =\sum_{n=0}^{\infty}\frac{(-1)^{n}c_{n+q}}{n!}z^{n}\left(  1-z\right)  ^{q}%
\end{align*}
Setting $z\rightarrow1-e^{-t}$ in the above, we have%
\[
\sum_{n=0}^{\infty}c_{n+q}\frac{\left(  e^{-t}-1\right)  ^{n}}{n!}%
e^{-tq}=q!\int\limits_{0}^{1}\binom{x}{q}e^{-xt}dx.
\]
Using the generating function of the $r$-Stirling numbers of the second kind
\cite{BRODER}
\[
\frac{\left(  e^{t}-1\right)  ^{n}}{n!}e^{tr}=\sum_{k=n}^{\infty}S_{r}\left(
k,n\right)  \frac{t^{k}}{k!}%
\]
gives%
\[
\sum_{n=0}^{\infty}c_{n+q}\frac{\left(  e^{-t}-1\right)  ^{n}}{n!}e^{-tq}%
=\sum_{k=0}^{\infty}\frac{\left(  -t\right)  ^{k}}{k!}\sum_{n=0}^{k}%
S_{q}\left(  k,n\right)  c_{n+q}.
\]
On the other hand, using (\ref{B2}) and Maclaurin's expansion of $e^{x},$ we
have%
\[
q!\int\limits_{0}^{1}\binom{x}{q}e^{-xt}dx=\sum_{k=0}^{\infty}\frac{\left(
-t\right)  ^{k}}{k!}\sum_{m=0}^{q}\frac{s\left(  q,m\right)  }{k+m+1}.
\]
Then comparing the coefficients of $\frac{\left(  -t\right)  ^{k}}{k!}$ yields%
\[
\sum_{n=0}^{k}S_{q}\left(  k,n\right)  c_{n+q}=\sum_{m=0}^{q}\frac{s\left(
q,m\right)  }{k+m+1}.
\]
Finally, utilizing the $r$-Stirling transform, given by
\[
a_{n}=\sum_{k=0}^{n}S_{r}\left(  n,k\right)  b_{k}\text{ }\left(
n\geq0\right)  \text{ if and only if }b_{n}=\sum_{k=0}^{n}s_{r}\left(
n,k\right)  a_{k}\text{ }\left(  n\geq0\right)  ,
\]
we obtain the following:

\begin{proposition}
For any non-negative integers $k$ and $q$%
\[
c_{k+q}=\sum_{n=0}^{k}\sum_{m=0}^{q}\frac{s_{q}\left(  k,n\right)  s\left(
q,m\right)  }{n+m+1}.
\]

\end{proposition}

\textbf{Example 5: }Now we evaluate an infinite series involving Cauchy
numbers with shifted indices.

\begin{proposition}
For every non-negative integer $q$%
\[
\sum_{n=0}^{\infty}\frac{\left(  -1\right)  ^{n}c_{n+q}}{n!\left(
n+q+1\right)  \cdots\left(  n+2q+1\right)  }=\sum_{j=0}^{q}\left(  -1\right)
^{q+j}\binom{q}{j}\binom{q+j}{j}\ln\left(  \frac{j+2}{j+1}\right)  .
\]

\end{proposition}

\begin{proof}
From (\ref{B6}), we have%
\begin{align*}
z^{q}\int\limits_{0}^{1}\binom{x}{q}(1-z)^{x}dx  &  =\sum_{n=q}^{\infty}%
\frac{(-1)^{n-q}c_{n}}{\left(  n-q\right)  !q!}z^{n}\left(  1-z\right)  ^{q}\\
&  =\sum_{n=0}^{\infty}\frac{(-1)^{n}c_{n+q}}{n!q!}z^{n+q}\left(  1-z\right)
^{q}.
\end{align*}
Integrating both sides of the above with respect to $z$ from $0$ to $1$ and
using well-known identity%
\[
B\left(  p,q\right)  =\int\limits_{0}^{1}z^{q-1}\left(  1-z\right)
^{p-1}dz=\frac{\Gamma\left(  p\right)  \Gamma\left(  q\right)  }{\Gamma\left(
p+q\right)  },
\]
we have%
\[
\sum_{n=0}^{\infty}\frac{\left(  -1\right)  ^{n}c_{n+q}}{n!\left(
n+q+1\right)  \cdots\left(  n+2q+1\right)  }=q!\int\limits_{0}^{1}\binom{x}%
{q}\frac{\Gamma\left(  x+1\right)  }{\Gamma\left(  x+q+2\right)  }dx.
\]
Utilizing (\ref{B2}), the properties $\Gamma\left(  x+1\right)  =x\Gamma
\left(  x\right)  $ and $\Gamma\left(  n+1\right)  =n!$ ($n\in%
%TCIMACRO{\U{2115} }%
%BeginExpansion
\mathbb{N}
%EndExpansion
$) and%
\begin{equation}
\frac{1}{\left(  x+1\right)  \cdots\left(  x+q+1\right)  }=\frac{1}{\left(
q+1\right)  !}\sum_{j=1}^{q+1}\left(  -1\right)  ^{j-1}\binom{q+1}{j}\frac
{j}{x+j}, \label{15}%
\end{equation}
we obtain%
\[
\sum_{n=0}^{\infty}\frac{\left(  -1\right)  ^{n}c_{n+q}}{n!\left(
n+q+1\right)  \cdots\left(  n+2q+1\right)  }=\frac{1}{\left(  q+1\right)
!}\sum_{j=1}^{q+1}\sum_{k=0}^{q}\left(  -1\right)  ^{j-1}\binom{q+1}%
{j}s\left(  q,k\right)  j\int\limits_{0}^{1}\frac{x^{k}}{x+j}dx.
\]
One can see that%
\[
\int\limits_{0}^{1}\frac{x^{k}}{x+j}dx=\left(  -1\right)  ^{k}j^{k}\ln\left(
\frac{j+1}{j}\right)  +\sum_{m=1}^{k}\binom{k}{m}\frac{\left(  -j\right)
^{k-m}}{m}\left(  \left(  j+1\right)  ^{m}-j^{m}\right)  .
\]
Then we have%
\begin{align*}
&  \sum_{n=0}^{\infty}\frac{\left(  -1\right)  ^{n}c_{n+q}}{n!\left(
n+q+1\right)  \cdots\left(  n+2q+1\right)  }\\
&  \qquad=\frac{1}{\left(  q+1\right)  !}\sum_{j=1}^{q+1}\left(  -1\right)
^{j+k-1}\binom{q+1}{j}\ln\left(  \frac{j+1}{j}\right)  \sum_{k=0}^{q}\left(
-1\right)  ^{q-k}s\left(  q,k\right)  j^{k+1}\\
&  \qquad+\frac{1}{\left(  q+1\right)  !}\sum_{j=1}^{q+1}\sum_{k=0}^{q}%
\sum_{m=1}^{k}\left(  -1\right)  ^{j+m-1}\binom{q+1}{j}s\left(  q,k\right)
\binom{k}{m}\frac{j^{k-m+1}}{m}\left(  \left(  j+1\right)  ^{m}-j^{m}\right)
.
\end{align*}
The second sum of the right-hand side is zero since%
\[
\sum_{k=1}^{n}\left(  -1\right)  ^{k}\binom{n}{k}k^{j}=0\text{ for }j<n.
\]
Moreover, using
\[
\sum_{k=0}^{q}\left(  -1\right)  ^{q-k}s\left(  q,k\right)  j^{k}=j\left(
j+1\right)  \cdots\left(  j+q-1\right)  ,
\]
we come to the desired result.
\end{proof}

\textbf{Example 6: }In this example we use the reciprocal binomial
coefficients $\binom{n+l}{l}^{-1},$ where $l$ is any non-negative integer. We
take $f\left(  x\right)  =\frac{1}{\left(  x+l+1\right)  \cdots\left(
x+l+r\right)  }$ and use \cite[Proposition 2]{LM},
\begin{equation}
\frac{1}{\left(  r-1\right)  !\left(  n+r\right)  \binom{n+l}{l}}=\sum
_{k=0}^{n}\binom{n}{k}\frac{\left(  -1\right)  ^{k}}{\left(  k+l+1\right)
\cdots\left(  k+l+r\right)  } \label{8}%
\end{equation}
and (\ref{15}). Thus (\ref{B3}) and (\ref{B4}) imply that

\begin{proposition}
For any integer $r\geq1$%
\begin{align}
\sum_{n=0}^{\infty}\frac{\left(  -1\right)  ^{n}c_{n}}{n!\left(  n+r\right)
\binom{n+l}{n}}  &  =\sum_{j=1}^{r}\left(  -1\right)  ^{j-1}\binom{r-1}%
{j-1}\ln\left(  \frac{l+j+1}{l+j}\right)  ,\nonumber\\
\sum_{n=m}^{\infty}\frac{\left(  -1\right)  ^{n-m}s\left(  n,m\right)
}{n!\left(  n+r\right)  \binom{n+l}{n}}  &  =\sum_{j=1}^{r}\binom{r-1}%
{j-1}\frac{\left(  -1\right)  ^{j-1}}{\left(  l+j\right)  ^{m+1}}. \label{L13}%
\end{align}
When $r=l+1$, these sums become%
\begin{align}
\sum_{n=0}^{\infty}\frac{\left(  -1\right)  ^{n}c_{n}}{\left(  n+l+1\right)
!}  &  =\frac{1}{l!}\sum_{j=0}^{l}\left(  -1\right)  ^{j}\binom{l}{j}%
\ln\left(  \frac{l+j+2}{l+j+1}\right)  ,\label{L4}\\
\sum_{n=m}^{\infty}\frac{\left(  -1\right)  ^{n-m}s\left(  n,m\right)
}{\left(  n+l+1\right)  !}  &  =\frac{1}{l!}\sum_{j=0}^{l}\binom{l}{j}%
\frac{\left(  -1\right)  ^{j}}{\left(  l+j+1\right)  ^{m+1}}. \label{L3}%
\end{align}

\end{proposition}

Setting $l=0$ in (\ref{L13}), we reach that
\begin{equation}
\sum_{n=m}^{\infty}\frac{\left(  -1\right)  ^{n-m}s\left(  n,m\right)
}{n!\left(  n+r\right)  }=\left(  -1\right)  ^{r+1}\left(  r-1\right)
!S\left(  -m,r\right)  , \label{L18}%
\end{equation}
where $S\left(  -n,r\right)  $ is the Stirling numbers of the second kind with
negative integral values, defined by \cite{BRANSON}
\begin{equation}
\frac{\left(  -1\right)  ^{r}}{r!}\sum_{j=1}^{r}\binom{r}{j}\frac{\left(
-1\right)  ^{j}}{j^{n}}=S\left(  -n,r\right)  . \label{L12}%
\end{equation}
With the use of \cite[Theorem 4]{WJ}, we can list some special cases as
follows:%
\[
S\left(  0,r\right)  =\frac{\left(  -1\right)  ^{r+1}}{r!},\text{ }S\left(
-1,r\right)  =\frac{\left(  -1\right)  ^{r+1}}{r!}H_{r},\text{ }S\left(
-2,r\right)  =\frac{\left(  -1\right)  ^{r+1}}{2r!}\left(  H_{r}^{2}%
+H_{r}^{\left(  2\right)  }\right)  .
\]

It is good to note that (\ref{L18}) is slightly different from \cite[Corollary
2.4]{XZZ}. (See \cite{A,J,WL,XYZ} for more examples of series involving
Stirling numbers of the first kind.)

\section{Series with skew-harmonic numbers.}

We work here with the skew-harmonic numbers
\[
H_{n}^{-}=1-\frac{1}{2}+\frac{1}{3}+...+\frac{(-1)^{n-1}}{n}\quad
(n\geq1),\quad H_{0}^{-}=0.
\]

\textbf{Example 7: }Applying the binomial formula \cite[Eq.(9.21)]{B2018}
\begin{equation}
\sum_{k=1}^{n}\binom{n}{k}(-1)^{k}\frac{1-2^{k}}{k}=H_{n}^{-} \label{B16}%
\end{equation}
we use the (entire) function
\[
f(x)=\frac{1-2^{x}}{x}=-\sum_{n=0}^{\infty}\frac{(\ln2)^{n+1}x^{n}}{(n+1)!}%
\]
where $f(0)=-\ln2$. With summation from $k=0$ the binomial formula (\ref{B16})
takes the form
\[
\sum_{k=0}^{n}\binom{n}{k}(-1)^{k}f(k)=-\ln2+H_{n}^{-}.
\]
Proposition A implies
\[
\int\limits_{0}^{1}\frac{1-2^{x}}{x}dx=\sum_{n=0}^{\infty}\,\frac
{(-1)^{n}c_{n}}{n!}\left\{  -\ln2+H_{n}^{-}\right\}  .
\]
The integral can be computed this way: with $t=x\ln2$
\[
\int\limits_{0}^{1}\frac{2^{x}-1}{x}dx=\sum_{n=1}^{\infty}\frac{(\ln2)^{n}%
}{n!n}\,=-Ein(-\ln2)
\]
where
\[
Ein(z)=\sum_{n=1}^{\infty}\frac{(-1)^{n-1}z^{n}}{n!n}%
\]
is the entire exponential integral function. This gives the evaluation

\begin{proposition}%
\begin{equation}
\sum_{n=0}^{\infty}\frac{(-1)^{n}c_{n}}{n!}\left\{  H_{n}^{-}-\ln2\right\}
=Ein(-\ln2). \label{B17}%
\end{equation}
Note that
\[
\lim_{n\rightarrow\infty}\left\{  H_{n}^{-}-\ln2\right\}  =0
\]
since $\lim_{n\rightarrow\infty}H_{n}^{-}=\ln2$.
\end{proposition}

\section{Series with hyperharmonic numbers}

In this section, we work with the hyperharmonic numbers which are defined by
the equation \cite{CG}
\begin{equation}
h_{n}^{\left(  r\right)  }=\binom{n+r-1}{r-1}\left(  H_{n+r-1}-H_{r-1}\right)
. \label{2}%
\end{equation}
Please see \cite{Benjamin,CaDa,CG,DilandMezo,DM} for more detail on
hyperharmonic numbers. \medskip

\textbf{Example 8: }Let $r$ be an integer $\geq1.$ We use the function
$f\left(  x\right)  =\frac{1}{\left(  x+1\right)  ^{2}\left(  x+2\right)
\cdots\left(  x+r\right)  }$ together with the binomial identity \cite[Theorem
3]{LM}%
\begin{equation}
\frac{h_{n+1}^{\left(  r\right)  }}{\left(  n+1\right)  \cdots\left(
n+r\right)  }=\sum_{k=0}^{n}\binom{n}{k}\frac{\left(  -1\right)  ^{k}}{\left(
k+1\right)  ^{2}\left(  k+2\right)  \cdots\left(  k+r\right)  } \label{4}%
\end{equation}
in Proposition A and Proposition B to obtain the following proposition.

\begin{proposition}%
\begin{equation}
\sum_{n=0}^{\infty}\frac{\left(  -1\right)  ^{n}c_{n}h_{n+1}^{\left(
r\right)  }}{\left(  n+r\right)  !}=\frac{1}{2\left(  r-1\right)  !}%
-\frac{r-1}{r!}\ln\left(  2\right)  +\frac{1}{r!}\sum_{j=2}^{r}\binom{r}%
{j}\frac{\left(  -1\right)  ^{j+1}}{j-1}\ln\left(  \frac{2j^{j}}{\left(
j+1\right)  ^{j}}\right)  \label{L2}%
\end{equation}
and%
\begin{equation}
\sum_{n=m}^{\infty}\frac{\left(  -1\right)  ^{n-m}s\left(  n,m\right)
h_{n+1}^{\left(  r\right)  }}{\left(  n+r\right)  !}=\left(  -1\right)
^{r+1}\sum_{k=0}^{m}S\left(  -k,r\right)  . \label{L6}%
\end{equation}

\end{proposition}

\begin{proof}
One have
\begin{align}
\frac{1}{\left(  x+1\right)  ^{2}\left(  x+2\right)  \cdots\left(  x+r\right)
}  &  =\frac{1}{\left(  r-1\right)  !\left(  x+1\right)  ^{2}}-\frac{r-1}%
{r!}\frac{1}{x+1}\nonumber\\
&  +\frac{1}{r!}\sum_{j=2}^{r}\binom{r}{j}\frac{\left(  -1\right)  ^{j+1}%
}{j-1}\left(  \frac{1}{x+1}-\frac{j}{x+j}\right)  . \label{11a}%
\end{align}
Using Proposition A in the above equation yields (\ref{L2}).

For the proof of (\ref{L6}), we first use
\[
\frac{1}{\left(  1+x\right)  ^{\alpha}}=\sum_{m=0}^{\infty}\left(  -1\right)
^{m}\binom{\alpha+m-1}{m}x^{m}%
\]
in (\ref{11a}) to obtain%
\begin{align*}
\sum_{n=m}^{\infty}\frac{\left(  -1\right)  ^{n-m}s\left(  n,m\right)
h_{n+1}^{\left(  r\right)  }}{\left(  n+r\right)  !}  &  =\frac{rm+1}%
{r!}+\frac{1}{r!}\sum_{j=2}^{r}\binom{r}{j}\frac{\left(  -1\right)  ^{j+1}%
}{j-1}\left(  1-\frac{1}{j^{m}}\right) \\
&  =\frac{rm+1}{r!}+\frac{1}{r!}\sum_{j=2}^{r}\binom{r}{j}\frac{\left(
-1\right)  ^{j+1}}{j^{m}}\left(  \frac{j^{m}-1}{j-1}\right)  .
\end{align*}
Then utilizing $\sum_{k=1}^{m}x^{k}=\left(  1-x^{m}\right)  /\left(
1-x\right)  $ in the above equation we have
\begin{align*}
\sum_{n=m}^{\infty}\frac{\left(  -1\right)  ^{n-m}s\left(  n,m\right)
h_{n+1}^{\left(  r\right)  }}{\left(  n+r\right)  !}  &  =\frac{rm+1}%
{r!}+\frac{1}{r!}\sum_{j=2}^{r}\binom{r}{j}\frac{\left(  -1\right)  ^{j+1}%
}{j^{m}}\sum_{k=0}^{m-1}j^{k}\\
&  =\frac{rm+1}{r!}+\frac{1}{r!}\left\{  -mr+\sum_{j=1}^{r}\binom{r}{j}%
\frac{\left(  -1\right)  ^{j+1}}{j^{m}}\sum_{k=0}^{m-1}j^{k}\right\} \\
&  =\frac{1}{r!}+\frac{1}{r!}\sum_{k=0}^{m-1}\sum_{j=1}^{r}\binom{r}{j}%
\frac{\left(  -1\right)  ^{j+1}}{j^{m-k}}.
\end{align*}
With the use of (\ref{L12}), we come to the desired result.
\end{proof}

It is good to note that the case $r=1$ of (\ref{L2}) and (\ref{L6}) are given
in \cite{B}. Moreover, setting $r=2$ in (\ref{L2}), $m=1$ and $m=2$ in
(\ref{L6}) give%
\[
\sum_{n=0}^{\infty}\frac{\left(  -1\right)  ^{n}c_{n}H_{n+2}}{\left(
n+1\right)  !}=\ln3-\ln2+\frac{1}{2},
\]%
\[
\sum_{n=1}^{\infty}\frac{h_{n+1}^{\left(  r\right)  }}{n\binom{n+r}{r}%
}=1+H_{r},\text{ \quad}\sum_{n=2}^{\infty}\frac{H_{n-1}h_{n+1}^{\left(
r\right)  }}{n\binom{n+r}{r}}=1+H_{r}+\frac{H_{r}^{2}+H_{r}^{\left(  2\right)
}}{2},
\]
respectively.

\textbf{Example 9: }Here we exploit the binomial formula \cite[Theorem 9]{LM}%
\begin{equation}
\frac{-h_{n}^{\left(  r\right)  }}{\left(  n+1\right)  \cdots\left(
n+r\right)  }=\sum_{k=0}^{n}\binom{n}{k}\frac{\left(  -1\right)  ^{k}}{\left(
k+1\right)  \cdots\left(  k+r\right)  }H_{k}, \label{9}%
\end{equation}
and the function $f\left(  x\right)  =\frac{\psi\left(  x+1\right)  +\gamma
}{\left(  x+1\right)  \cdots\left(  x+r\right)  }.$ With the use of (\ref{15})
and
\begin{equation}
\psi(x+1)+\gamma=\sum_{n=1}^{\infty}(-1)^{n+1}\zeta(n+1)x^{n},\quad|x|\,<1.
\label{L9}%
\end{equation}
we have%
\[
\frac{\psi\left(  x+1\right)  +\gamma}{\left(  x+1\right)  \cdots\left(
x+r\right)  }=\frac{1}{r!}\sum_{j=1}^{r}\left(  -1\right)  ^{j-1}\binom{r}%
{j}\sum_{n=1}^{\infty}\left(  -1\right)  ^{n+1}x^{n}\left(  \sum_{k=1}%
^{n}\frac{\zeta\left(  k+1\right)  }{j^{m-k}}\right)
\]
Then using (\ref{L12}) and (\ref{B4}) give the following evaluation.

\begin{proposition}%
\begin{equation}
\sum_{n=m}^{\infty}\frac{\left(  -1\right)  ^{n-m}s\left(  n,m\right)
h_{n}^{\left(  r\right)  }}{\left(  n+r\right)  !}=\left(  -1\right)
^{r+1}\sum_{k=1}^{m}S\left(  k-m,r\right)  \zeta\left(  k+1\right)  .
\label{L5}%
\end{equation}

\end{proposition}

The case $r=1$ was discussed in \cite{B2018a}. Setting $m=1$ in (\ref{L5})
gives \cite[Eq.(27)]{DB}. Moreover, $m=2$ and $m=3$ in (\ref{L5}) yield%
\begin{align*}
\sum_{n=2}^{\infty}\frac{H_{n-1}h_{n}^{\left(  r\right)  }}{n\binom{n+r}{r}}
&  =H_{r}\frac{\pi^{2}}{6}+\zeta\left(  3\right)  ,\\
\sum_{n=3}^{\infty}\frac{(H_{n-1}^{2}-H_{n-1}^{(2)})h_{n}^{\left(  r\right)
}}{n\binom{n+r}{r}}  &  =\left(  H_{r}^{2}+H_{r}^{\left(  2\right)  }\right)
\frac{\pi^{2}}{6}+2H_{r}\zeta\left(  3\right)  +\frac{\pi^{4}}{45}%
\end{align*}
respectively.

\textbf{Example 10: }In this example we take $f\left(  x\right)  =\frac
{\psi\left(  x+1\right)  +\gamma}{\left(  x+1\right)  ^{2}\left(  x+2\right)
\cdots\left(  x+r\right)  }$ and use the identity \cite[Theorem 11]{LM}%
\begin{align*}
&  \sum_{k=0}^{n}\binom{n}{k}\frac{\left(  -1\right)  ^{k+1}}{\left(
k+1\right)  \left(  k+1\right)  ^{\left(  r\right)  }}H_{k}\\
&  \qquad\quad=\frac{1}{2\left(  n+1\right)  \left(  r-1\right)  !}\left\{
\left(  H_{n+r}-H_{r-1}\right)  ^{2}-\left(  H_{n+r}^{\left(  2\right)
}-H_{r-1}^{\left(  2\right)  }\right)  \right\}  .
\end{align*}
Using (\ref{11a}) and (\ref{L9}), we have%
\begin{align*}
&  \frac{\psi\left(  x+1\right)  +\gamma}{\left(  x+1\right)  ^{2}\left(
x+2\right)  \cdots\left(  x+r\right)  }\\
&  \quad=\frac{1}{\left(  r-1\right)  !}\left[  \sum_{m=1}^{\infty}\left(
-1\right)  ^{m-1}x^{m}\sum_{k=1}^{m}\left(  m-k+1\right)  \zeta\left(
k+1\right)  \right] \\
&  \quad+\frac{1-r}{r!}\left[  \sum_{m=1}^{\infty}\left(  -1\right)
^{m-1}x^{m}\sum_{k=1}^{m}\zeta\left(  k+1\right)  \right] \\
&  \quad+\frac{1}{r!}\sum_{j=2}^{r}\binom{r}{j}\frac{\left(  -1\right)
^{j+1}}{j-1}\left[  \sum_{m=1}^{\infty}\left(  -1\right)  ^{m-1}x^{m}%
\sum_{k=1}^{m}\zeta\left(  k+1\right)  \right] \\
&  \quad+\frac{1}{r!}\sum_{j=2}^{r}\binom{r}{j}\frac{\left(  -1\right)  ^{j}%
}{\left(  j-1\right)  }\left[  \sum_{m=1}^{\infty}\left(  -1\right)
^{m-1}x^{m}\right]  \left[  \sum_{k=1}^{m}\frac{\zeta\left(  k+1\right)
}{j^{m-k}}\right]  .
\end{align*}
From (\ref{B4}) and some arrangements we obtain%
\begin{align*}
&  \sum_{n=m}^{\infty}\frac{\left(  -1\right)  ^{n-m+1}s\left(  n,m\right)
\left[  \left(  H_{n+r}-H_{r-1}\right)  ^{2}-\left(  H_{n+r}^{\left(
2\right)  }-H_{r-1}^{\left(  2\right)  }\right)  \right]  }{2\left(
n+1\right)  !}\\
&  \qquad=\frac{1}{r}\sum_{k=1}^{m}\zeta\left(  k+1\right)  +\frac{1}{r}%
\sum_{k=1}^{m}\zeta\left(  k+1\right)  \sum_{j=1}^{r}\binom{r}{j}\frac{\left(
-1\right)  ^{j+1}}{j^{m-k}}\left(  1+j+\cdots+j^{m-k-1}\right)  .
\end{align*}
Then using (\ref{L12}) yields the following:

\begin{proposition}%
\begin{align}
&  \sum_{n=m}^{\infty}\frac{\left(  -1\right)  ^{n-m+1}s\left(  n,m\right)
\left[  \left(  H_{n+r}-H_{r-1}\right)  ^{2}-\left(  H_{n+r}^{\left(
2\right)  }-H_{r-1}^{\left(  2\right)  }\right)  \right]  }{2\left(
n+1\right)  !}\nonumber\\
&  \qquad\qquad\qquad\qquad\qquad\qquad\qquad=\frac{1}{r}\sum_{k=0}^{m-1}%
A_{k}\left(  r\right)  \zeta\left(  m-k+1\right)  \label{L10a}%
\end{align}
where
\[
A_{k}\left(  r\right)  =1+\left(  -1\right)  ^{r+1}r!\sum_{l=1}^{k}S\left(
-l,r\right)  .
\]

\end{proposition}

Since $S\left(  -l,1\right)  =1,$ for $r=1$ in (\ref{L10a}), we have
\begin{equation}
\sum_{n=m}^{\infty}\frac{\left(  -1\right)  ^{n-m}s\left(  n,m\right)  \left[
H_{n+1}^{\left(  2\right)  }-H_{n+1}^{2}\right]  }{2\left(  n+1\right)
!}=\sum_{k=1}^{m}\left(  m-k+1\right)  \zeta\left(  k+1\right)  . \label{L11}%
\end{equation}
On the other hand, we have \cite[Eq.(22)]{B}%
\[
\sum_{n=m}^{\infty}\frac{\left(  -1\right)  ^{n-m}s\left(  n,m\right)  \left[
H_{n+1}^{2}+H_{n+1}^{\left(  2\right)  }\right]  }{2\left(  n+1\right)
!}=\left(  m+1\right)  \left(  m+2\right)  .
\]
Combining this result with (\ref{L11}), we reach that%
\begin{align*}
\sum_{n=m}^{\infty}\frac{\left(  -1\right)  ^{n-m}s\left(  n,m\right)
H_{n+1}^{\left(  2\right)  }}{2\left(  n+1\right)  !}  &  =\left(  m+1\right)
\left(  m+2\right)  +\sum_{k=1}^{m}\left(  m-k+1\right)  \zeta\left(
k+1\right)  ,\\
\sum_{n=m}^{\infty}\frac{\left(  -1\right)  ^{n-m}s\left(  n,m\right)
H_{n+1}^{2}}{2\left(  n+1\right)  !}  &  =\left(  m+1\right)  \left(
m+2\right)  -\sum_{k=1}^{m}\left(  m-k+1\right)  \zeta\left(  k+1\right)
\end{align*}
Setting $m=1,2,3$ in the above give%
\begin{align*}
\sum_{n=1}^{\infty}\frac{H_{n+1}^{\left(  2\right)  }}{n\left(  n+1\right)  }
&  =12+\frac{\pi^{2}}{3},\text{ \ \ }\sum_{n=1}^{\infty}\frac{H_{n+1}^{2}%
}{n\left(  n+1\right)  }=12-\frac{\pi^{2}}{3},\\
\sum_{n=2}^{\infty}\frac{H_{n-1}H_{n+1}^{\left(  2\right)  }}{n\left(
n+1\right)  }  &  =24+\frac{2\pi^{2}}{3}+2\zeta\left(  3\right)  ,\\
\sum_{n=2}^{\infty}\frac{H_{n-1}H_{n+1}^{2}}{n\left(  n+1\right)  }  &
=24-\frac{2\pi^{2}}{3}-2\zeta\left(  3\right) \\
\sum_{n=3}^{\infty}\frac{(H_{n-1}^{2}-H_{n-1}^{(2)})H_{n+1}^{\left(  2\right)
}}{n\left(  n+1\right)  }  &  =80+2\pi^{2}+8\zeta\left(  3\right)  +\frac
{2\pi^{4}}{45},\\
\sum_{n=3}^{\infty}\frac{(H_{n-1}^{2}-H_{n-1}^{(2)})H_{n+1}^{2}}{n\left(
n+1\right)  }  &  =80-2\pi^{2}-8\zeta\left(  3\right)  -\frac{2\pi^{4}}{45}.
\end{align*}

\textbf{Example 11. }Let $r$ be a integer $>1.$ Then we use the identity
\cite{DB}
\[
\frac{-h_{n}^{\left(  r\right)  }}{\binom{n+r-1}{n}^{2}}=\sum_{k=0}^{n}\left(
-1\right)  ^{k}\binom{n}{k}\frac{k}{\left(  k+r-1\right)  ^{2}}%
\]
together with the function $f\left(  x\right)  =\frac{x}{\left(  x+r-1\right)
^{2}}.$ For $\left\vert x\right\vert <\left\vert r-1\right\vert $ we have the
Taylor series%
\[
f\left(  x\right)  =\sum_{m=0}^{\infty}\frac{\left(  -1\right)  ^{m}\left(
m+1\right)  }{\left(  r-1\right)  ^{m+2}}x^{m+1}.
\]
From (\ref{B3}) and (\ref{B4}) we find the representations below:

\begin{proposition}
Let $r$ be an integer $>1.$ Then we have%
\begin{align}
\sum_{n=0}^{\infty}\frac{(-1)^{n+1}c_{n}h_{n}^{\left(  r\right)  }}%
{n!\binom{n+r-1}{n}^{2}}  &  =\ln\left(  \frac{r}{r-1}\right)  -\frac{1}%
{r},\label{L19}\\
\sum_{n=m}^{\infty}\frac{(-1)^{n-m}s(n,m)h_{n}^{\left(  r\right)  }}%
{n!\binom{n+r-1}{n}^{2}}  &  =\frac{\left(  m+1\right)  }{\left(  r-1\right)
^{m+2}}. \label{L20}%
\end{align}

\end{proposition}

For $m=1$ and $m=2$ in (\ref{L20}) we find that
\[
\sum_{n=1}^{\infty}\frac{h_{n}^{\left(  r\right)  }}{n\binom{n+r-1}{n}^{2}%
}=\frac{2}{\left(  r-1\right)  ^{3}}\text{ \quad and\quad}\sum_{n=2}^{\infty
}\frac{H_{n-1}h_{n}^{\left(  r\right)  }}{n\binom{n+r-1}{n}^{2}}=\frac
{3}{\left(  r-1\right)  ^{4}}.
\]

At the end of this section we want to note that if we apply (\ref{B3}) to the
functions and binomial identities given in Examples 9 and 10, we come to very
challenging integrals.
\[
\sum_{n=0}^{\infty}\frac{\left(  -1\right)  ^{n+1}c_{n}h_{n}^{\left(
r\right)  }}{\left(  n+r\right)  !}=\int\limits_{0}^{1}\frac{\psi\left(
x+1\right)  +\gamma}{\left(  x+1\right)  \cdots\left(  x+r\right)  }dx,
\]
and
\begin{align*}
&  \frac{1}{\left(  r-1\right)  !}\sum_{n=0}^{\infty}\frac{\left(  -1\right)
^{n+1}c_{n}\left\{  \left(  H_{n+r}-H_{r-1}\right)  ^{2}-\left(
H_{n+r}^{\left(  2\right)  }-H_{r-1}^{\left(  2\right)  }\right)  \right\}
}{\left(  n+1\right)  !\left(  n+r\right)  !}\\
&  \qquad\qquad\qquad\qquad\qquad\qquad\qquad\qquad\qquad=\int\limits_{0}%
^{1}\frac{\psi\left(  x+1\right)  +\gamma}{\left(  x+1\right)  ^{2}\left(
x+2\right)  \cdots\left(  x+r\right)  }dx.
\end{align*}
The first integral when $r=1$ is
\[
\int\limits_{0}^{1}\frac{\psi\left(  x+1\right)  +\gamma}{x+1}dx.
\]
Multiplying (\ref{L9}) by the geometric series $\left(  1+x\right)  ^{-1}%
=\sum_{n=0}^{\infty}\left(  -1\right)  ^{n}x^{n},$ we come to the expansion%
\[
\frac{\psi\left(  x+1\right)  +\gamma}{x+1}=\sum_{n=1}^{\infty}x^{n}\left\{
\left(  -1\right)  ^{n-1}\sum_{k=1}^{n}\zeta\left(  k+1\right)  \right\}  .
\]
So we have
\[
\int\limits_{0}^{1}\frac{\psi\left(  x+1\right)  +\gamma}{x+1}dx=\sum
_{n=1}^{\infty}\left\{  \frac{\left(  -1\right)  ^{n-1}}{n+1}\sum_{k=1}%
^{n}\zeta\left(  k+1\right)  \right\}  \approx0.3606201929\text{.}%
\]
(This complements the results in \cite{B2018a}).

\end{document}